\documentclass[12pt]{amsart}
\usepackage{amsfonts}
\usepackage{amsmath}
\usepackage{amssymb,latexsym}
\usepackage[mathcal]{eucal}
\usepackage{amscd}
\input xy
\xyoption{all}

\newcommand{\Z}{\mathbb{Z}}
\newcommand{\R}{\mathbb{R}}
\newcommand{\Ga}{\Gamma}
\newcommand{\Sig}{\Sigma}
\newcommand{\ThetaGrp}{\Theta}
\newcommand{\Del}{\Delta}
\newcommand{\del}{\delta}
\newcommand{\Id}{\text{Id}}

\newcommand{\Prol}{{\rm{Prol\,}}}
\newcommand{\comm}{{\rm{Comm\,}}}

\newcommand{\imp}{\Rightarrow}

\theoremstyle{plain}
\newtheorem{thm}{Theorem}[section]
\newtheorem{prop}[thm]{Proposition}
\newtheorem{lem}[thm]{Lemma}
\newtheorem{cor}[thm]{Corollary}

\theoremstyle{definition}
\newtheorem{defn}[thm]{Definition}
\newtheorem{exas}[thm]{Examples}
\newtheorem{rmk}[thm]{Remark}

\numberwithin{thm}{section}

\begin{document}

\title{On jointly transitive commuting minimal homeomorphisms}

\dedicatory{Dedicated to Benjy Weiss, a teacher and friend.}

\author{Eli Glasner}
\address{Department of Mathematics\\
Tel Aviv University\\
Tel Aviv\\
Israel}
\email{glasner@math.tau.ac.il}
\date{October 25, 2026}

\begin{abstract}
In analogy with ergodic theory, we introduce the notions of joint topological transitivity and joint minimality
for commuting minimal homeomorphisms on a compact metric space $X$.
We prove an analogue of the
Berend--Bergelson theorem for the first notion and investigate to what extent this theorem holds for the second.
On the way we show that 
if $T$ and $S$ are two minimal distal commuting homeomorphisms of a compact metric space $X$,
the system $(X, \Gamma)$, where $\Gamma$ is the group generated by $T$ and $S$, is distal. 
\end{abstract}



\maketitle

\section{Introduction}

In ergodic theory, a result by Berend and Bergelson \cite{BB-84} establishes that for commuting measure-preserving transformations $T$ and $S$, the sequence of iterates $(T^n, S^n)$ is jointly ergodic if and only if two conditions are met: the product transformation $T \times S$ is ergodic, and the  transformation $T^{-1}S$ is ergodic.
In this note, we investigate an analogue of this phenomenon in topological dynamics.

We first restrict our attention to the case of two commuting homeomorphisms $T$ and $S$,
each acting minimally on a compact metric space $X$.
The topological counterpart to joint ergodicity is \emph{joint topological transitivity} (or $\Delta$-transitivity), which requires the diagonal orbit  $\{(T \times S)^n(x,x) : n \in \Z\}$ to be dense in $X \times X$, for  a dense $G_\del$ set of points $x \in X$.

A central observation of this note is that the topological analogue of the Berend-Bergelson theorem holds: we will show that joint topological transitivity is equivalent to the combined topological transitivity of both the product system $(X \times X, T \times S)$ and the system $(X, T^{-1}S)$.

We then explore the much more rigid property of \emph{joint minimality}, identifying where the transitivity implications break down, and describing the necessary counterexamples.

Thank are due to the  referees for their helpful comments.

\section{Definitions and Main Results}

Let $X$ be a compact (metric) space and let $T$ and $S$ be two commuting self-homeomorphisms of $X$ such that both actions $(X,T)$ and $(X,S)$ are minimal.

Let $\Ga =\langle T, S \rangle$ be the group generated by $T$ and $S$, and denote $R = T^{-1}S$.
On the product space $X \times X$, let $\Sig =\langle T \times T , S \times S, T \times S \rangle= \langle \Id \times R , T \times T , S \times S \rangle$ and let $\ThetaGrp = \langle T \times T, S \times S \rangle$.
Set $\Del = \{(x,x) : x \in X\}$.

\begin{defn}
Consider the following dynamical conditions for the system:
\begin{enumerate}
\item[(0)] The group $\Sig$ acts minimally on $X \times X$.
\item[(1)] The system $(X,R)$ is topologically transitive.
\item[(2)] $\overline{\cup\{ (T \times S)^n \Del : n \in \Z\}}= X \times X$.
\item[(3)] The system $(X \times X, T \times S)$ is topologically transitive.
\item[$(1^*)$] The system $(X,R)$ is minimal.
\item[$(2^*)$] For every $x \in X$, $ \overline{\{ (T \times S)^n (x,x) : n \in \Z\}}= X \times X$.
\item[$(2^{**})$] For every $x$ in a dense $G_\delta$ subset $X_0 \subset X$, we have $ \overline{\{ (T \times S)^n (x,x) : n \in \Z\}}= X \times X$.
\item[$(3^*)$] The system $(X \times X, T \times S)$ is minimal (i.e., $T$ and $S$ are disjoint).
 \end{enumerate}
\end{defn}

The following theorem summarizes the logical structure of these conditions, explicitly establishing the topological analogue of the Berend-Bergelson theorem in the middle row.

\begin{thm}\label{MainTheorem}
The following implications and equivalences hold:

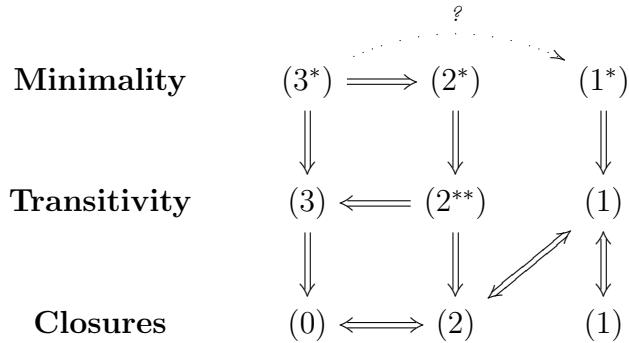
\begin{equation*}
\xymatrix{
\text{\bf Minimality} & (3^*) \ar@{=>}[r] \ar@{=>}[d] & (2^*) \ar@{=>}[d] & (1^*) \ar@{=>}[dd] \\
\text{\bf Transitivity} & (3) \ar@{<=}[r] & (2^{**}) \ar@{=>}[d] &  \\
\text{\bf Closures} & (0) \ar@{<=>}[r] & (2) \ar@{<=>}[r] & (1)
}
\end{equation*}

Furthermore, the topological analogue of the Berend-Bergelson theorem is satisfied: 
$$ (2^{**}) \iff (3) + (1) $$
\end{thm}

In Section 4 we deal with the distal case and show that if $T$ and $S$ are commuting distal minimal homeomorphisms, then the equivalences $(3) \iff (3^*)$ and $(1) \iff (1^*)$ hold unconditionally.

\begin{rmk}
A preliminary draft of this note had been sitting in my drawer for many years. I recently decided to revive it, and after considerable elaboration, brought it to its present form. After posting the manuscript to the arXiv, I was informed by Song Shao of the recent paper \cite{DKW-25}.
While this latter work proves the same main theorem as the one presented here, the methods of proof are very different. Furthermore, I believe that the study of joint minimality, the treatment of the distal case, and several of the counterexamples presented in this note are new.
\end{rmk}

\section{Proofs of the Implications}

\begin{defn}
Let $(X,T)$ be a (compact metric) cascade.
The {\em prolongation relation} on $X$ is the subset
$$
Prol_T(X) = \{(x,x') \in X \times X : \exists x_i \to x \ \ {\text{and}} \ \ \ \exists \ {\text{a sequence}} \ n_i \in \mathbb{Z} \ {\text{such that }} \  T^{n_i}x_i \to x'\}.
$$ 
See \cite{AG-88} for more details  on this relation.
\end{defn}

The following proposition is well known;
we present its proof anyhow.

\begin{prop}\label{prop-prol-trans}
Let $(X,T)$ be a dynamical system on a compact metric space.
The system $(X,T)$ is topologically transitive if and only if 
$$
\Prol_T(x) = \{x' \in X : (x,x') \in  Prol_T(X) \} =X,
$$ 
for every $x \in X$.
\end{prop}

\begin{proof}
Assume first that $(X,T)$ is topologically transitive. Let $x$ and $y$ be arbitrary points in $X$.
Choose a sequence of open neighborhoods $\{U_k\}_{k=1}^\infty$ forming a local base at $x$ (so $\bigcap_{k \geq 1} U_k = \{x\}$) and a sequence of open neighborhoods $\{V_k\}_{k=1}^\infty$ forming a local base at $y$.
By the topological transitivity, there exists an integer $n_k$ such that $T^{n_k} U_k \cap V_k \neq \emptyset$.
Choose a point $x_k \in U_k$ such that $T^{n_k} x_k \in V_k$.
The sequences $x_k \to x$ and $T^{n_k}x_k \to y$ show that $y \in \Prol_T(x)$. 
Thus,  $\Prol_T(x) = X$.

Conversely, assume that $\Prol_T(x) = X$ for every $x \in X$.
Let $U, V \subset X$ be two nonempty open sets. Let $x \in U$ and $x' \in V$.
By assumption there are sequences 
$x_i \to x$ and   $n_i \in \Z$ such that $ T^{n_i}x_i \to x'$.
Hence for large enough $i$,
$ T^{n_i}U \cap V \neq \emptyset$.
\end{proof}

\begin{proof}[Proof of Theorem \ref{MainTheorem}]
The implications $(i^*) \imp (i)$ are obvious, as are $(3^*) \imp (2^*)$, $(2^*) \imp (2^{**})$, and $(2^{**}) \imp (2)$.
\vspace{.3cm}

We now establish the equivalences and remaining implications step-by-step.

\vspace{2mm}
\noindent \textbf{The Lower Cycle: $(0) \iff (2) \iff (1)$}

\noindent $(2) \iff (1)$: We have 
$$ 
\overline{\cup\{ (T \times S)^n \Del : n \in \Z\}}= \overline{\cup\{ (\Id \times R)^n \Del : n \in \Z\}},
$$
hence (2) is the same as $ \overline{\cup\{ (\Id \times R)^n \Del : n \in \Z\}}= X \times X$. However, clearly 
\[
\overline{\cup\{ (\Id \times R)^n \Del : n \in \Z\}} =\cup\{\{x\}\times \Prol_R(x) : x \in X\}. 
\]

The right-hand side is
$$
\{(x, y) : y \in \Prol_R(x)\}.
$$
Thus,  (2) is equivalent to $\Prol_R(x) =X$ for
every $x$, and by Proposition 3.2 this is equivalent to topological transitivity of $(X,R)$.

(0) $\imp$ (2):\ 
Since $\Del$ is invariant under both $T \times T$ and $S \times S$ we have
$\Sigma \Del = \cup\{ (T \times S)^n \Del : n \in \Z\}$, whence
$X \times X = \overline{\Sigma \Del} = 
\overline{\cup\{ (T \times S)^n \Del : n \in \Z\}}$.

(2) $\imp$ (0):\ 
Let $E=E(X\times X, \Sig)$ be the enveloping semigroup of the
system $(X\times X, \Sig)$.
We have $E(X\times X, \Theta)
\subset E$, where $E(X\times X, \Theta)$ denotes the enveloping 
semigroup of the system $(X \times X, \Theta)$.
Let $u \in E(X\times X, \Theta)$ be a minimal idempotent
of the semigroup $E(X\times X, \Theta)$ and let $I \subset Eu$ 
be a minimal left $E$-ideal.
Finally let
$v$ be a minimal idempotent in $I$. We will show that $u \in I$.
Note that $vu =v$.

Next let $\pi_i : X \times X \to X, \ i=1,2$
denote the projections onto the first and second coordinates,
respectively.
These are homomorphisms of $\Sig$ dynamical systems
where on both coordinates $\Sig$ acts as $\Ga$ in an obvious way.
The projection maps $\pi_i$ canonically induce semigroup homomorphisms
$\pi_1^* : E(X \times X,\Sig) \to E(X,\Gamma)$ and
$\pi_2^* : E(X \times X,\Sig) \to E(X,\Ga)$.
It follows that $u_i = \pi_i^* u$ and $v_i = \pi_i^* v,\ i=1,2$
are minimal idempotents in their respective enveloping semigroups.
We also set $I_i =  \pi_i^*(I)$ and note that the $I_i$'s are minimal left ideals
in their corresponding enveloping semigroups.

The equations $v_i u_i = v_i,\ i=1,2$, hold in $I_i$
and it follows that $v_i$ is in the same $E(X,\Ga)$ minimal ideal as $u_i$.
From Ellis' theory it then follows that $u_i v_i = u_i, \ i=1,2$.
Now observe that
these latter equations imply the equality $uv=u$ as claimed.
Also notice that, in turn, this implies that $u \in I$ and that in particular
$u$ is also a minimal idempotent of $E = E(X \times X, \Sig)$.

To conclude the proof recall that, as the system $(X,\Ga)$ is minimal,
picking an arbitrary $x\in X$, there is a minimal idempotent $u
\in E(X,\Ga)$ with $u x =x$.
Since $E(X,\Ga)$ can be identified
with $E(X \times X, \Theta)$ we also have that $u$ is a minimal idempotent
in $E(X\times X,\Theta)$ and $u(x,x) = (x,x)$.
However the fact that $u$ is also a minimal idempotent in $E
= E(X \times X, \Sig)$ ensures that $(x,x)$ is a minimal point for
the $\Sig$ action.

Since $(X, \Gamma)$ is minimal, the diagonal $\Delta$ is contained in 
$\overline{\Sigma(x,x)}$.  Indeed, the diagonal $\Gamma$-orbit of $(x,x)$ is dense in $\Delta$.
As $\overline{\Sigma(x,x)}$ is closed and $\Sigma$-invariant, it contains
$$
\overline{\Sigma\Delta} = \overline{\bigcup \{(T \times S)^n \Delta : n \in \Z\}}.
$$

By condition  (2), this last set is $X \times X$. Hence
$\overline{\Sigma \Delta} = X \times X$.
Since $(x,x)$ is also minimal, $X \times X$ itself is a minimal $\Sigma$-system.
So, finally $\overline{\Sig(x,x)} = X \times X$ is $\Sig$ minimal as claimed.

\vspace{2mm}
\noindent \textbf{The Berend--Bergelson Analogue: $(2^{**}) \iff (3) + (1)$}

\textbf{Direction $(2^{**}) \imp (3) + (1)$:} \\
Assume $(2^{**})$. The existence of a point $x \in X$ for which the orbit $\{ (T \times S)^n (x,x) : n \in \Z\}$ is dense in $X \times X$ trivially implies that the system $(X \times X, T \times S)$ possesses a dense orbit. Thus, $(X \times X, T \times S)$ is topologically transitive, yielding (3).
Furthermore, since $(2^{**}) \imp (2)$ and the lower cycle establishes (2) $\iff$ (1), condition (1) follows.

\textbf{Direction $(3) + (1) \imp (2^{**})$:} \\
Assume both (3) and (1) hold. 
We proceed by utilizing the Baire Category Theorem. 
For any pair of nonempty open sets $U, V \subset X$, define the set of visiting times:
$$
N(U, V) = \{ x \in X : \exists n \in \Z \text{ such that } T^n x \in U \text{ and } S^n x \in V \}
$$
Clearly, $N(U, V) = \bigcup_{n \in \Z} (T^{-n}U \cap S^{-n}V)$, which is an open subset of $X$. We must show it is dense.

Let $W \subset X$ be an arbitrary nonempty open set. We seek an integer $n \in \Z$ and a point $x \in W$ such that $T^n x \in U$ and $S^n x \in V$.
Using the transformation $R = T^{-1}S$, we have $S^n = (TR)^n = T^n R^n$ (since $T$ and $S$ commute). Thus, the requirement is equivalent to finding $x \in W$ such that $T^n x \in U$ and $T^n R^n x \in V$.

Let $\Delta_W = \{(x,x) : x \in W\}$ be the diagonal segment over $W$ in $X \times X$. By condition (1), the system $(X, R)$ is topologically transitive.

By condition (3), the system $(X \times X, T \times S)$ is topologically transitive, implying that the open, $(T \times S)$-invariant set $O = \bigcup_{n \in \Z} (T \times S)^{-n}(U \times V)$ is dense in $X \times X$. We must show that $\Delta_W \cap O \neq \emptyset$. 

Assume for contradiction that $\Delta_W \cap O = \emptyset$. Because $O$ is open and dense, its complement $O^c$ is closed and nowhere dense. By our assumption and the  $(T \times S)$-invariance of $O^c$, we have $\bigcup_{m \in \Z} (T \times S)^m \Delta_W \subset O^c$. 

Since the system $(X, T)$ is minimal and $W$ is open, the collection $\{T^k W\}_{k \in \Z}$ is an open cover of the compact space $X$. Thus, there exists a finite subcover, meaning $X = \bigcup_{i=1}^p T^{k_i} W$ for some finite set of integers $\{k_1, \dots, k_p\}$. This implies $\Delta_X = \bigcup_{i=1}^p (T \times T)^{k_i} \Delta_W$. 

Furthermore, because $T \times T$ commutes with $T \times S$, we have:
$$ \bigcup_{m \in \Z} (T \times S)^m \Delta_X = \bigcup_{i=1}^p (T \times T)^{k_i} \left( \bigcup_{m \in \Z} (T \times S)^m \Delta_W \right) \subset \bigcup_{i=1}^p (T \times T)^{k_i} O^c. $$

By the equivalence $(1) \iff (2)$ established in Theorem \ref{MainTheorem}, the set on the left side is dense in $X \times X$. However, the set on the right is a finite union of closed, nowhere dense sets, and is therefore itself closed and nowhere dense. The closure of the left side is $X \times X$, which implies $X \times X$ is contained in a nowhere dense set---a clear contradiction. Thus, $\Delta_W \cap O \neq \emptyset$, which guarantees there exists some $n \in \Z$ such that $(T \times S)^n \Delta_W \cap (U \times V) \neq \emptyset$.

This implies there is an $x \in W$ satisfying $(T^n x, S^n x) \in U \times V$, meaning $N(U, V) \cap W \neq \emptyset$.
Therefore, $N(U, V)$ is a dense open set in $X$.

Choosing countable topological bases $\{U_k\}$ and $\{V_m\}$ for $X$, the Baire Category Theorem implies that the countable intersection
$$X_0 = \bigcap_{k, m} N(U_k, V_m)$$
is a dense $G_\delta$ subset of $X$. For every $x \in X_0$, the diagonal orbit $\{ (T \times S)^n (x,x) : n \in \Z\}$ intersects every basic open set in $X \times X$, and is therefore dense. This establishes $(2^{**})$.
\end{proof}

\section{The Distal Case}

Let $T$ and $S$ be commuting minimal homeomorphisms of a compact metric
space $X$, and put $\Gamma=\langle T,S\rangle$.  In this section we prove
that if both $(X,T)$ and $(X,S)$ are distal, then the joint $\Gamma$-action
is distal.  Consequently, in the distal case topological transitivity can
be upgraded to minimality in conditions $(1)$ and $(3)$.

We use the canonical Furstenberg tower of one of the two transformations.
The point of using the canonical tower is that every automorphism of the
system preserves all its stages.  
We do not need, and do not assert, that
the canonical towers associated with $T$ and $S$ coincide.

Recall the following definition:

\begin{defn}
The extension $\pi:(X,S)\to(Y,S)$ is \emph{isometric} if there is a
continuous relative metric
\[
 \rho:R_\pi\longrightarrow[0,\infty),
 \qquad
 R_\pi=\{(x,x'):\pi(x)=\pi(x')\},
\]
whose restriction to each fibre is a compatible metric and which satisfies
\[
 \rho(Sx,Sx')=\rho(x,x')
 \qquad ((x,x')\in R_\pi).
\]
\end{defn}

The following proposition is a well known folklore result. We provide a proof for completeness.
\begin{prop}[Maximal isometric intermediate factor]
\label{prop:maximal-isometric}
Let
\[
q:(X,T)\longrightarrow(Y,T)
\]
be an extension of minimal compact metric systems. Then there exists an
intermediate factorization
\[
X\xrightarrow{p_{\mathrm{iso}}}Z_{\mathrm{iso}}
 \xrightarrow{\pi_{\mathrm{iso}}}Y,
\qquad
q=\pi_{\mathrm{iso}} \circ p_{\mathrm{iso}},
\]
with the following properties:
\begin{enumerate}
\item $\pi_{\mathrm{iso}}:(Z_{\mathrm{iso}},T)\to(Y,T)$ is isometric;
\item if
\[
X\xrightarrow{p}Z\xrightarrow{\pi}Y
\]
is any intermediate factorization for which $\pi$ is isometric, then
there is a unique factor map
\[
\theta:Z_{\mathrm{iso}}\longrightarrow Z
\]
such that
\[
p=\theta \circ p_{\mathrm{iso}}
\qquad\text{and}\qquad
\pi_{\mathrm{iso}}=\pi \circ \theta.
\]
\end{enumerate}
Consequently, $Z_{\mathrm{iso}}$ is unique, as an intermediate factor
of $X$ over $Y$. It is called the maximal isometric intermediate
factor, or the maximal relatively equicontinuous factor of $q$.
\end{prop}

\begin{proof}
For every isometric intermediate factorization
\[
X\xrightarrow{p_i}Z_i\xrightarrow{\pi_i}Y,
\]
let
\[
R_i=R_{p_i}
   =\{(x,x')\in X^2:p_i(x)=p_i(x')\}.
\]
Each $R_i$ is a closed $T\times T$-invariant equivalence relation and
satisfies
\[
R_i\subset R_q.
\]
Set
\[
R_{\mathrm{iso}}=\bigcap_i R_i.
\]
This is again a closed $T\times T$-invariant equivalence relation
contained in $R_q$. Define
\[
Z_{\mathrm{iso}}=X/R_{\mathrm{iso}}.
\]
Since $R_{\mathrm{iso}}\subset R_q$, the map $q$ factors uniquely as
\[
X\xrightarrow{p_{\mathrm{iso}}}Z_{\mathrm{iso}}
 \xrightarrow{\pi_{\mathrm{iso}}}Y.
\]

It remains to verify that $\pi_{\mathrm{iso}}$ is isometric. Since
$X$ is metrizable, the open set
\[
X^2\setminus R_{\mathrm{iso}}
   =\bigcup_i(X^2\setminus R_i)
\]
is Lindel\"of. Thus there is a countable family
\[
R_1,R_2,\ldots
\]
such that
\[
R_{\mathrm{iso}}=\bigcap_{n=1}^{\infty}R_n.
\]
Consider the diagonal map
\[
P:X\longrightarrow\prod_{n=1}^{\infty}Z_n,
\qquad
P(x)=(p_n(x))_{n\geq1}.
\]
Its kernel relation is precisely $R_{\mathrm{iso}}$. Hence
$Z_{\mathrm{iso}}$ is naturally homeomorphic to the compact image
$P(X)$.

All the coordinate maps $\pi_n:Z_n\to Y$ have the same value on
$P(X)$, so $P(X)$ is contained in the fibre product
\[
\prod_Y Z_n
 =
 \left\{(z_n):
          \pi_n(z_n)=\pi_m(z_m)
          \text{ for all }n,m\right\}.
\]
Choose bounded continuous $T$-invariant relative metrics $\rho_n$ for
the extensions $\pi_n$. On each fibre of the product extension define
\[
\rho\big((z_n),(z_n')\big)
 =
 \sum_{n=1}^{\infty}
 2^{-n}\rho_n(z_n,z_n').
\]
After replacing each $\rho_n$ by
\(\min\{1,\rho_n\}\), if necessary, this series converges uniformly.
It defines a compatible continuous relative metric on the fibre
product. Moreover,
\[
\rho\big(T(z_n),T(z_n')\big)
 =
 \sum_{n=1}^{\infty}
 2^{-n}\rho_n(Tz_n,Tz_n')
 =
 \rho\big((z_n),(z_n')\big).
\]
Restricting this relative metric to the closed invariant subspace
$P(X)\cong Z_{\mathrm{iso}}$ shows that
\(\pi_{\mathrm{iso}}:Z_{\mathrm{iso}}\to Y\) is isometric.

Finally, for every isometric intermediate factor $p_i:X\to Z_i$, we
have
\[
R_{\mathrm{iso}}\subset R_i.
\]
Therefore $p_i$ factors uniquely through $p_{\mathrm{iso}}$, giving
the required map
\[
\theta_i:Z_{\mathrm{iso}}\longrightarrow Z_i.
\]
This proves maximality. If two intermediate factors have this
universal property, each factors through the other; the resulting
maps are mutually inverse because both commute with the corresponding
factor maps from $X$. Thus the maximal intermediate factor is unique.
\end{proof}

\begin{defn}
Let  $\pi:(X,S)\to(Y,S)$ be an extension of dynamical systems. The {\em relative regionally proximal relation}
is define on $X \times X$ as follows:
$$
Q_T(\pi) = \{(x,y)  : \exists x_i \to x, \exists y_i \to y, \exists n_i \in \mathbb{Z}, \exists z \in X,
\ {\text{such that}}\ 
t_i(x_i,y_i) \to (z,z)\}.
$$
\end{defn}

The {\em canonical Furstenberg distal tower} in the next theorem refers to the tower introduced  by Furstenberg  
in his celebrated paper  \cite{Fur-63}. The isometric extensions in this tower are always the maximal ones.

Recall the standard facts that  a factor of a distal system is distal 
(see e.g. \cite{Gl-76}).

\begin{lem}[Canonicity of the distal tower]\label{lem:canonical-tower}
Let $(X,T)$ be a minimal distal system, and let
\[
 (X_\alpha,T_\alpha)_{\alpha\leq\eta}
\]
be its canonical Furstenberg tower, regarded as a tower of factors of
$(X,T)$.  If $S\in\operatorname{Aut}(X,T)$, then $S$ induces a
homeomorphism $S_\alpha$ on every $X_\alpha$, and
\[
 T_\alpha S_\alpha=S_\alpha T_\alpha .
\]
Every bonding map in the tower is therefore equivariant for the action of
$\langle T,S\rangle$.
\end{lem}

\begin{proof}
We argue by transfinite induction.  The assertion is trivial for the
one-point factor $X_0$.  


Suppose inductively that $S$ acts on $X_\alpha$, and write
\[
q_\alpha:X\longrightarrow X_\alpha,
\qquad
q_\alpha S=S_\alpha q_\alpha.
\]
Let $Q_T(q_\alpha)$ denote the relative regionally proximal relation
of the extension $q_\alpha$.

We claim that
\[
(S\times S)Q_T(q_\alpha)=Q_T(q_\alpha).
\]
Indeed, suppose $(x,x')\in Q_T(q_\alpha)$. Then there are sequences
$x_i\to x$, $x_i'\to x'$, and $n_i\in\mathbb Z$ such that
$q_\alpha(x_i)=q_\alpha(x_i')$
and
$d(T^{n_i}x_i,T^{n_i}x_i')\longrightarrow 0$.
Since $q_\alpha S=S_\alpha q_\alpha$, we have
\[
q_\alpha(Sx_i)
 =S_\alpha q_\alpha(x_i)
 =S_\alpha q_\alpha(x_i')
 =q_\alpha(Sx_i').
\]
Moreover, using $ST=TS$,
\[
T^{n_i}Sx_i=ST^{n_i}x_i,
\qquad
T^{n_i}Sx_i'=ST^{n_i}x_i'.
\]
By uniform continuity of $S$,
\[
d(T^{n_i}Sx_i,T^{n_i}Sx_i')\longrightarrow0.
\]
Therefore
\[
(Sx,Sx')\in Q_T(q_\alpha).
\]
This proves
\[
(S\times S)Q_T(q_\alpha)\subset Q_T(q_\alpha).
\]
Applying the same argument to $S^{-1}$ gives the reverse inclusion.

Let $E_T(q_\alpha)$ be the relative equicontinuous structure relation,
that is, the smallest closed $T\times T$-invariant equivalence relation
containing $Q_T(q_\alpha)$. Since $S\times S$ preserves
$Q_T(q_\alpha)$ and commutes with $T\times T$, the relation
\[
(S\times S)E_T(q_\alpha)
\]
is again a closed $T\times T$-invariant equivalence relation containing
$Q_T(q_\alpha)$. By the minimality of $E_T(q_\alpha)$,
\[
E_T(q_\alpha)
 \subset (S\times S)E_T(q_\alpha).
\]
Applying the same argument to $S^{-1}$ yields the reverse inclusion.
Thus
\[
(S\times S)E_T(q_\alpha)=E_T(q_\alpha).
\]

The quotient
\[
X/E_T(q_\alpha)\longrightarrow X_\alpha
\]
is the canonical maximal $T$-isometric intermediate extension of
$q_\alpha$. Since $E_T(q_\alpha)$ is $S\times S$-invariant, $S$
descends to this quotient.

At a limit ordinal $\lambda$, the factor $X_\lambda$ is the inverse limit
of the preceding factors.  The compatible family
$(S_\alpha)_{\alpha<\lambda}$ therefore induces a homeomorphism
$S_\lambda$ of $X_\lambda$.  Commutation with $T_\alpha$ follows at every
stage from commutation on $X$.
\end{proof}

\begin{lem}[Isometric in one direction]\label{lem:relative-distal}
Let $T$ and $S$ commute on a compact metric space $W$ and descend to
commuting homeomorphisms on a compact metric space $Y$.  Write
$\Gamma=\langle T,S\rangle$, and suppose that
\[
 \pi:(W,T,S)\longrightarrow(Y,T,S)
\]
is equivariant.  Assume that
\begin{enumerate}
\item $\pi:(W,T)\to(Y,T)$ is an isometric extension;
\item $(W,S)$ is distal; and
\item $(Y,\Gamma)$ is distal.
\end{enumerate}
Then $(W,\Gamma)$ is distal.
\end{lem}

\begin{proof}
Let $\rho$ be a continuous compatible relative metric on
\[
 R_\pi=\{(w,w')\in W\times W:\pi(w)=\pi(w')\}
\]
which is invariant under $T\times T$.  Suppose that $w,w'\in W$ are
$\Gamma$-proximal.  Since $W$ is metrizable, there are sequences
$n_i,m_i\in\mathbb Z$ such that
\[
 d\bigl(T^{n_i}S^{m_i}w,T^{n_i}S^{m_i}w'\bigr)\longrightarrow0.
\]
After applying $\pi$, the points $\pi(w)$ and $\pi(w')$ are
$\Gamma$-proximal in $Y$.  The distality of $(Y,\Gamma)$ therefore gives
$\pi(w)=\pi(w')$.

The pairs $(S^{m_i}w,S^{m_i}w')$ consequently belong to $R_\pi$.  By the
$T$-invariance of the relative metric,
\begin{align*}
 \rho\bigl(S^{m_i}w,S^{m_i}w'\bigr)
 &=\rho\bigl(T^{n_i}S^{m_i}w,T^{n_i}S^{m_i}w'\bigr)
 \longrightarrow0.
\end{align*}
Here the last implication follows from the continuity of $\rho$ and the
compactness of the diagonal in $W\times W$.  Since $\rho$ is a compatible
relative metric and its zero set is the diagonal, compactness of $R_\pi$
now implies
\[
 d\bigl(S^{m_i}w,S^{m_i}w'\bigr)\longrightarrow0.
\]
Thus $w$ and $w'$ are proximal for $S$.  The distality of $(W,S)$ gives
$w=w'$, proving that the $\Gamma$-action on $W$ is distal.
\end{proof}

\begin{prop}\label{prop-distal-joint}
Let $T$ and $S$ be commuting minimal homeomorphisms of a compact metric
space $X$.  If $(X,T)$ and $(X,S)$ are distal, then the joint system
$(X,\Gamma)$, where $\Gamma=\langle T,S\rangle$, is distal.
\end{prop}

\begin{proof}
Let
\[
 (X_\alpha,T_\alpha)_{\alpha\leq\eta},\qquad
 X_0=\{*\},\quad X_\eta=X,
\]
be the canonical Furstenberg tower of the minimal distal system $(X,T)$.
Thus, at a successor ordinal, the bonding map
\[
 \pi_\alpha:X_{\alpha+1}\longrightarrow X_\alpha
\]
is isometric for $T$, and at a limit ordinal the corresponding system is
the inverse limit of the preceding systems.  By
Lemma~\ref{lem:canonical-tower}, $S$ acts on every $X_\alpha$ and all the
bonding maps are $\Gamma$-equivariant.

We prove by transfinite induction that $(X_\alpha,\Gamma)$ is distal.  This
is trivial for $X_0$.  Suppose that $(X_\alpha,\Gamma)$ is distal.  The
system $(X_{\alpha+1},S_{\alpha+1})$ is a factor of the minimal distal
system $(X,S)$ and is therefore distal.  Since
$X_{\alpha+1}\to X_\alpha$ is $T$-isometric,
Lemma~\ref{lem:relative-distal} shows that
$(X_{\alpha+1},\Gamma)$ is distal.

If $\lambda$ is a limit ordinal and $(X_\alpha,\Gamma)$ is distal for every
$\alpha<\lambda$, then their inverse limit $X_\lambda$ is distal: a
$\Gamma$-proximal pair in $X_\lambda$ projects to a proximal, and hence
equal, pair in every $X_\alpha$; the inverse-limit coordinates separate
points.  The induction reaches $X_\eta=X$, and hence $(X,\Gamma)$ is
distal.
\end{proof}

\begin{cor}\label{cor:distal-minimality}
If $T$ and $S$ are commuting distal minimal homeomorphisms, then the
equivalences
\[
 (3)\iff(3^*)
 \qquad\text{and}\qquad
 (1)\iff(1^*)
\]
hold unconditionally.
\end{cor}

\begin{proof}
By Proposition~\ref{prop-distal-joint}, the joint action of
$\Gamma=\langle T,S\rangle$ on $X$ is distal.  Hence every element of
$\Gamma$, and in particular $R=T^{-1}S$, acts distally on $X$.

The homeomorphism $T\times S$ is distal on $X\times X$.  Indeed, if two
points of $X\times X$ were proximal under $T\times S$, their first
coordinates would be proximal under $T$ and their second coordinates
would be proximal under $S$; distality in the two coordinates forces the
two points to be equal.

A topologically transitive distal homeomorphism of a compact space is
minimal: every point in a distal system is almost periodic, so the orbit
closure of a transitive point is a minimal set and must be the whole
space.  Therefore, (3) implies $(3^*)$, and (1) implies $(1^*)$.  The
reverse implications are immediate.
\end{proof}

\begin{rmk}\label{rmk:no-common-tower}
The proof establishes that the canonical $T$-tower is preserved by $S$ and
that it becomes a tower of distal $\Gamma$-extensions.  It does not show
that a $T$-isometric successor extension is $S$-isometric, nor does it show
that the canonical Furstenberg towers of $(X,T)$ and $(X,S)$ coincide.
\end{rmk}

\begin{rmk}

IN \cite{SX-25} Shao and Xu show that
if $(X ,S)$ and $(X ,T )$ are minimal systems with commuting $S$ and $T$, then they possess an identical higher-order
regionally proximal relation
\[
 RP^{[d]}(X,T)=RP^{[d]}(X,S)
 \qquad\text{for every }d\in\mathbb N.
\]
Consequently, both $(X ,S)$ and $(X ,T )$ share the same increasing
sequence of pro-nilfactors.
\end{rmk}

 \section{Counterexamples and Limitations}
 
 In this section we examine to what extent the above implications cannot be reversed.
\begin{exas}
\begin{enumerate}
\item \textbf{Showing $(1^*) \not\imp (2^*)$ and $(2) \not\imp (3)$:} 
Let $X = S^1$, $T =R_\alpha$, and $S =R_{2\alpha}$ for an irrational $\alpha$.
$T^{-1}S = R_\alpha$ is minimal (satisfying $1^*$). Here, 
$$
\overline{\cup \{(T \times S)^n \Del : n \in \Z\}}= \{(x, x+z) : x, z \in X\} = X \times X,
$$
 satisfying (2).
However, for every $x \in X$, 
 $$
 \overline {\{(T^n x, S^n x) : n \in \Z\}} =\{(u,v): 2u - v = x\} \neq X \times X,
 $$ failing $(2^*)$ and $(3)$.

\item \textbf{Showing $(1^*) + (3) \not\imp (2^*)$ and $(2^{**}) \not\imp (2^*)$:}
For the relevant facts about horocycle flows see e.g. \cite[Chapter 4]{Gl-03}.

Let $\Lambda$ be an arithmetic uniform lattice in $G=SL(2,\R)$ and let $X = G/\Lambda$.
Choose a hyperbolic $c \in \comm(\Lambda) \setminus \Lambda$, conjugate it to 
$d = \left(\begin{smallmatrix} a & 0\\ 0 & a^{-1} \end{smallmatrix}\right)$ with $a \ne 1$, so 
$dg=gc$ for some $g$.

Since $\Lambda$ is arithmetic, its commensurator group $\comm(\Lambda)$ is dense and contains hyperbolic elements (trace outside $[-2,2]$), making $c$ diagonalizable over $\R$ and conjugate to $d$ via some $g \in G$ (so $dg = gc$).
Set $T = \left( \begin{smallmatrix} 1 & 1\\ 0 & 1 \end{smallmatrix} \right)$ and $S = d^{-1}Td = \left( \begin{smallmatrix} 1 & a^{-2}\\ 0 & 1 \end{smallmatrix} \right)$.
$R = T^{-1}S = \left(\begin{smallmatrix} 1 & a^{-2} -1\\ 0 & 1 \end{smallmatrix} \right)$ acts minimally on $X$ (satisfying $1^*$).
Because $T$ and $S$ are non-trivial horocycle flows, they are strongly mixing.
The product of strongly mixing systems is strongly mixing, ensuring that $(X \times X, T \times S)$ is topologically transitive.
Therefore, condition (3) holds, which by Theorem \ref{MainTheorem} guarantees $(2^{**})$ holds.

However, for $x = g\Lambda$, $(T^n x, S^nx) = (\Id \times d^{-1})(T \times T)^n (g\Lambda, g c\Lambda)$, meaning the orbit closure is trapped in a finite union of Hecke correspondences and is not $X \times X$.
Thus, $(2^*)$ fails.

\item \textbf{Showing $(2^*) \not\imp (3^*)$ for Commuting Systems:} 
By Theorem \ref{thm-HSY22} below, there exists a minimal weakly mixing system $(X,T)$ where, for the commuting pair $T$ and $S = T^2$, the diagonal orbit of \emph{every} $x \in X$ is dense in $X \times X$.
This example satisfies condition $(2^*)$. However,
the authors also prove that no such point is minimal under $T \times T^2$, 
so  that, in particular, the product system $(X \times X, T \times T^2)$ is not minimal.
Thus, $(3^*)$ fails even for commuting transformations.

\item \textbf{Showing $(2^*) \not\imp (3^*)$ for Non-Commuting Systems:} 
If $T$ and $S$ do not commute, joint density does not force disjointness.
Let $T$ be a doubly minimal homeomorphism on the Cantor set \cite{W}.
Using the Baire category theorem, find a homeomorphism $L$ such that for all $x$, $Lx$ is not on the $T$-orbit of $x$.
Defining $S = L^{-1}TL$ ensures the $S \times T$ orbit of $(x,x)$ is dense, satisfying $(2^*)$, but $T$ and $S$ are not disjoint.

\item
Theorem \ref{thm-HSY22} below gives another example showing that
\((2^{**})\not\Rightarrow(3^*)\).

\item
The following example is due to Hui Xu and, independently, to Chunlin Liu. 
I thank them for the permission to present it here.

Recall that a dynamical system  $(Y,f)$ is called {\em proximal orbit dense}, or  POD,
if $(Y,f)$  is totally minimal and whenever $x, y \in Y$ are distinct points, then
for some $n \not = 0$, the points $f^nx$ and $y$ are proximal (see also \cite{KO-12}). 

Let $(Y,f)$ be a POD system. Let $X=Y\times Y$, $T=f \times f^2$, $S=f^3\times f^4$.
Then $R=T^{-1}S=f^2\times f^2$, and $(X,R)$ is not minimal. But since $(Y,f)$ is POD, it follows from the 
properties of such systems as shown in \cite{FKS}, that
$(X\times X, T\times S)=(Y^4,f\times f^2\times f^3\times f^4)$ is minimal.

Thus, in general, $(3^*) \not \imp (1^*)$.
\end{enumerate}
\end{exas}
 
\section{$\Delta$-Transitivity and Polynomials}
The property of $\Delta$-transitivity corresponds to condition $(2^{**})$.

\begin{defn}
\begin{enumerate}
\item
Let $(X,T)$ be a minimal (compact metric) cascade.
We say that a point $x \in X$ is $d$ {\em $\Delta$-transitive}
if $\overline{\{(T^n x, T^{2n} x, \dots, T^{dn}x) : n \in \Z\}}= X \times X \times \cdots \times X$.
The system $(X,T)$ is {\em $\Delta$-transitive}, if the set of points $x\in X$
which are $\Del$-transitive forms a dense $G_\del$ subset of $X$.
\item
A point $x \in X$ is $d$ {\em $\Del$-minimal}, if the orbit closure
$$
\overline{\{(T \times T^2 \times \cdots \times T^d)^n (x,x, \dots, x) : n \in \Z\}},
$$ 
is a minimal subset of the system $(X^d, T \times T^2 \times \cdots \times T^d)$.
\end{enumerate}
\end{defn}

The following theorems indicate when these conditions and their related notions naturally occur.

\begin{thm}[Glasner, 1994]\label{thm-Gl}\cite{Gl-94}
If a topological dynamical system $(X, T)$ is minimal and weakly mixing, then for every integer $d \ge 2$, it is $\Delta$-transitive.
Specifically, there exists a dense $G_\delta$ subset $X_0 \subseteq X$ such that for every $x \in X_0$, the orbit $\{(T^n x, T^{2n}x, \dots, T^{dn}x) : n \in \mathbb{Z}\}$ is dense in $X^d$.
\end{thm}

\begin{thm}[Zhang, and Zhao, 2021]  \cite{ZZ-21}
If $(X, T)$ is minimal and weakly mixing, and $p_1, \dots, p_d$ are essentially distinct generalized polynomials taking integer values, there exists a dense $G_\delta$ subset $X_0 \subseteq X$ such that for every $x \in X_0$, the orbit $\{(T^{p_1(n)}x, \dots, T^{p_d(n)}x) : n \in \mathbb{Z}\}$ is dense in $X^d$.
\end{thm}

\begin{thm}[Huang, Shao, and Ye, 2022]\label{thm-HSY22} \cite{HSY22}
There is a minimal weakly mixing system $(X,T)$ without any multiply minimal points.
In fact, for all $x \in X$, $(x,x)$ is a transitive point of $(X \times X, T \times T^2)$ but not minimal \cite{HSY22}.
This demonstrates that $\Delta$-transitivity $(2^{**})$ does not guarantee the existence of multiply minimal points.
\end{thm}

\section{Higher-Dimensional Analogues}

The topological equivalences established for two transformations extend naturally to any finite sequence of commuting minimal homeomorphisms.
Let $T_1, \dots, T_d$ be commuting minimal homeomorphisms of $X$.
Let $\tau = T_1 \times T_2 \times \cdots \times T_d$ acting on the product space $X^d$, and let $\Delta_d = \{(x, x, \dots, x) : x \in X\} \subset X^d$ be the diagonal.
We define the $d$-dimensional analogues of our core transitivity conditions:
\begin{itemize}
    \item[$(1_d)$] For every $i \neq j$, the system $(X, T_i^{-1}T_j)$ is topologically transitive.
    \item[$(2_d)$] $\overline{\bigcup_{n \in \Z} \tau^n \Delta_d} = X^d$.
    \item[$(3_d)$] The system $(X^d, \tau)$ is topologically transitive.
    \item[$(2_d^{**})$] There exists a dense $G_\del$ subset 
    $    X_0 \subset X$ such that for every $x \in X_0$, 
    $$
    \overline{\{\tau^n(x, \dots, x) : n \in \Z\}} = X^d
    $$
\end{itemize}

\begin{thm}\label{thm-higher-dim}
For any $d \ge 2$, the topological analogue of the Berend-Bergelson theorem holds:
$$ 
 (2_d^{**}) \iff (1_d) + (3_d).
 $$
As in the two-dimensional case, topological transitivity of the product already forces the pairwise difference systems to be topologically transitive.
\end{thm}

The proof is similar to that of the case $d=2$ so we skip the details.

\vspace{.3cm}

For the specific multiple recurrence case involving linear iterates of a single transformation ($\tau = T \times T^2 \times \cdots \times T^d$), the structural rigidity allows us to immediately deduce full topological transitivity on the product space.

\begin{thm}\label{thm-modest}
Let $(X, T)$ be a minimal weakly mixing system. For any $d \ge 2$, let $\tau = T \times T^2 \times \cdots \times T^d$.
Then the product system $(X^d, \tau)$ is topologically transitive.
\end{thm}

\begin{proof}
By Theorem \ref{thm-Gl}, because $(X, T)$ is minimal and weakly mixing, it is $\Delta$-transitive.
This guarantees the existence of a dense $G_\del$ subset of points $x \in X$ whose diagonal orbit under $\tau$ is dense in $X^d$.
This is exactly condition $(2_d^{**})$. By Theorem \ref{thm-higher-dim}, condition $(2_d^{**})$ implies condition $(3_d)$, establishing the topological transitivity of $(X^d, \tau)$.
(Note that condition $(1_d)$ is also trivially satisfied here: since $T_i = T^i$, $T_i^{-1}T_j = T^{j-i}$. Because $(X,T)$ is weakly mixing, it is totally transitive, ensuring every non-zero power $T^{j-i}$ is topologically transitive).
\end{proof}

\bibliographystyle{amsplain}

\end{document}